\documentclass[12pt]{amsart}

\usepackage{graphicx}
\usepackage{amsmath,amsthm}
\usepackage{amsfonts}
\usepackage{amssymb}
\usepackage{a4wide}
\title[]{On the Kurosh problem for algebras over a general field}

\author{Jason P.~Bell}
\thanks{The first-named author was supported by NSERC grant 31-611456.}

\subjclass[2000]{}

\address{Jason Bell\\
Department of Mathematics\\
Simon Fraser University\\
Burnaby, BC V5A 1S6\\
Canada}

\email{jpb@math.sfu.ca}

\author{Alexander A. Young}
\thanks{The research of the second-named author was partially supported by the United States National Science Foundation.}
\address{Alexander Young\\
Department of Mathematics\\
University of California at San Diego\\
9500 Gilman Drive\\
La Jolla, CA 92093-0112\\
USA}

\email{aayoung@math.ucsd.edu}


\newtheorem{thm}{Theorem}[section]
\newtheorem{lem}[thm]{Lemma}

\newtheorem{prop}[thm]{Proposition}
\newtheorem{cor}[thm]{Corollary}

\theoremstyle{definition}
\newtheorem{defn}[thm]{Definition}

\newtheorem{quest}{Question}

\newtheorem{notn}[thm]{Notation}

\begin{document}

\begin{abstract} Smoktunowicz, Lenagan, and the second-named author have recently given an example of a nil algebra of Gelfand-Kirillov dimension at most three.  Their construction requires a countable base field, however.  We show that for any field $k$ and any monotonically increasing function $f(n)$ which grows super-polynomially but subexponentially there exists an infinite-dimensional finitely generated nil $k$-algebra whose growth is asymptotically bounded by $f(n)$.  This construction gives the first examples of nil algebras of subexponential growth over uncountable fields.
\end{abstract}
\maketitle
\section{Introduction}
In recent years there has been renewed interest in the construction of finitely generated algebraic algebras that are not finite-dimensional \cite{BS, BSS, LS, LSY, Sm1, Sm2, Sm3}.  The first such examples were constructed by Golod and Sharafervich \cite{Go, GS} using a combinatorial criterion that guaranteed that algebras with certain presentations are infinite-dimensional.  Their construction provided a counter-example to Kurosh's conjecture, which asserted that finitely generated algebraic algebras should be finite-dimensional over their base fields.  By modifying their construction, Golod and Shafarevich were also able to give a solution to the celebrated Burnside problem, which is the group-theoretic analogue of the Kurosh conjecture and asked whether finitely generated torsion groups are necessarily finite.  

The connection between Burnside-type problems in the theory of groups and Kurosh-type problems in ring theory has led to many interesting conjectures in both fields, which have arisen naturally from results in one field or the other.  

One of the fascinating developments in the theory of groups has been Gromov's theorem, which states that a finitely generated group of polynomial bounded growth (see Section $2$ for precise definitions) is nilpotent-by-finite; that is, it has a normal nilpotent subgroup of finite index.  As finitely generated nilpotent torsion groups are finite, Gromov's theorem immediately gives the non-trivial result that a finitely generated torsion group of polynomial bounded growth is finite.  

In light of this result, it was natural to ask whether a corresponding result held for rings.  Specifically, Small asked whether a finitely generated algebraic algebra of polynomially bounded growth should be finite-dimensional \cite{LSY}.  Surprisingly, Lenagan and Smoktunowicz \cite{LS} showed that this is not the case, by constructing a finitely generated nil algebra of Gelfand-Kirillov dimension at most $20$.  Their construction only works over countable base fields, however.  Recently, Lenagan, Smoktunowicz, and the second-named author \cite{LSY} have shown that the bound on Gelfand-Kirillov dimension can be replaced by $3$, a substantial improvement.  On the other hand, it is known that the bound cannot be made lower than $2$, as finitely generated algebras of Gelfand-Kirillov dimension strictly less than two satisfy a polynomial identity \cite{SSW}, \cite[Theorem 2.5, p. 18]{KL} and the Kurosh conjecture holds for the class of algebras satisfying a polynomial identity \cite[Section 6.4]{He}.

The fact that these constructions do not work over an uncountable base field is not surprising, as many results have appeared over the years which show there is a real dichotomy that exists regarding Kurosh-type problems when one considers base fields.  For example, algebraic algebras over uncountable fields have the \emph{linearly bounded degree} property; that is, given a fixed finite-dimensional subspace of an algebraic algebra over an uncountable field, there is a natural number $d$, depending on the subspace, such that all elements in this subspace have degree at most $d$.  On the other hand, Smoktunowicz \cite{Sm1} has given an example of a nil algebra over a countable base field with the property that the ring of polynomials over this algebra is not nil and hence this algebra cannot have linearly bounded degree.  

This distinction, and the fact that the elements in a finitely generated algebra over a countable base field can be enumerated, has led to a relative dearth of interesting examples of algebraic algebras over uncountable base fields.  Indeed, over uncountable fields there has not been much progress since the original construction of Golod and Shafarevich.

Our main result is to show that for every field $K$, there exists a finitely generated algebraic $K$-algebra with subexponential growth.  More specifically, we prove the following result.
\begin{thm} \label{thm: main1}
Let $K$ be a field and let 
$\alpha:[0,\infty)\to [0,\infty)$ be a weakly increasing function tending to $\infty$.  Then there is a finitely generated connected graded $K$-algebra $$B=\bigoplus_{n\ge 0} B(n)$$ such that
the homogeneous maximal ideal, $\oplus_{n\ge 1} B(n)$, is nil and ${\rm dim}(B(n))\le n^{\alpha(n)}$ for all sufficiently large $n$.
\end{thm}
We recall that a $K$-algebra $B=\oplus_n B(n)$ is \emph{connected} if $B(0)=K$ and it is \emph{graded} if each $B(n)$ is a $K$-vector space and $B(n)B(m)\subseteq B(n+m)$ for each $(n,m)\in \mathbb{N}^2$.  

Equivalently, Theorem \ref{thm: main1} says that if $\beta(n)$ is any monotonically increasing function that grows subexponentially but superpolynomially in $n$,  then we can find a connected graded $K$-algebra $B$ whose homogeneous maximal ideal is nil and has the property that the coefficients of its Hilbert series are eventually less than $\beta(n)$.  We note that we do not acquire lower bounds on the growth and there is no reason that our constructions could not in fact have polynomially bounded growth.  We suspect, however, that the growth is indeed superpolynomial.  One should contrast this situation with the situation in group theory, where considerably less is known about the possible growth types of finitely generated torsion groups of superpolynomial growth.  There have been many constructions of \emph{branch groups}, which provide examples of groups that have subexponential but superpolynomial growth.  The first such construction was done by Grigorchuk \cite{Gri}, and estimates of Bartoldi \cite{Bar} show that the growth of Branch group constructions is at least $\exp(\sqrt{n})$, but significantly less than $\exp(cn)$ for any $c > 0$.

We use the methods from the corresponding constructions done over countable base fields \cite{LS, LSY}.  The main differences in our construction is that our algebras must have linearly bounded degree and the elements of our algebras cannot be enumerated.  To get around this, we take a finite-dimensional subspace $V$ that contains $1$ and generates our algebra.  We then write our algebra as a countable union of the nested finite-dimensional subspaces $V^n$ and for each such subspace we give a finite set of relations which implies that each element of the subspace is algebraic.  To get subexponential growth, it is necessary to choose the relations efficiently.  There does not appear to be a way of improving our choice of relations to give an infinite-dimensional finitely generated algebraic algebra of polynomially bounded growth over an uncountable base field.

The outline of this paper is as follows.  In Section $2$, we give background on growth of algebras.  In Section $3$, we prove some general results about the growth of algebras with certain presentations.  In Section $4$, we consider the problem of linearly bounded degree and prove Theorem \ref{thm: main1}.
\section{Growth of groups and algebras}
In this section, we recall the basic definitions of growth, which we will use throughout this paper.
Let $G$ be a finitely generated group with generating set $S=\{g_1,\ldots ,g_d\}$.  Then each element of $G$ can be expressed as a word in $g_1, g_2, \ldots ,g_d$ and $g_1^{-1},\ldots ,g_d^{-1}$.  This gives us a weakly increasing function 
\begin{equation} d_S:\mathbb{N}\to \mathbb{N},\end{equation}
in which $d_S(n)$ is the number of distinct words in $g_1^{\pm 1},\ldots ,g_d^{\pm 1}$ of length at most $n$. 
We call $d_S$ the \emph{growth function of} $G$ \emph{with respect to the generating set} $S$.

While this growth function obviously depends on our choice of generating set $S$, it can be shown that, in some sense, the asymptotic behavior of $d_S$ is independent of this choice \cite{KL}.  More specificially, we say that two functions 
$f,g:\mathbb{N}\to \mathbb{N}$ are \emph{asymptotically equivalent} if there exist positive integers $m_0$ and $m_1$ and positive constants $C_0$ and $C_1$ such that
\begin{equation} \label{eq: 1} f(n) \le C_0g(m_0n) \qquad {\rm for~ all~ sufficiently~ large~} n\end{equation}
and
\begin{equation} \label{eq: 2} g(n) \le C_1f(m_1n) \qquad {\rm for~ all~ sufficiently~ large~} n.\end{equation}
In the case that two functions $f(n)$ and $g(n)$ are asymptotically equivalent, we will write
\begin{equation}
f(n) \asymp g(n).
\end{equation}
(We note that this is slightly different from the notion of two functions of $n$ being asymptotic to each other, which means that their ratio tends to $1$ as $n\to \infty$.  For example, $e^{an}\asymp e^{bn}$ for any $a,b>0$.)
In the case that Equation (\ref{eq: 1}) holds but Equation (\ref{eq: 2}) does not necessarily hold, we say that $f(n)$ is \emph{asymptotically dominated} by $g(n)$ and write
\begin{equation} f(n) \preceq g(n).
\end{equation}
We note that if $S$ and $T$ are two finite generating sets for a finitely generated group $G$, then there exist natural numbers $m_0$ and $m_1$ such that
$$T\cup T^{-1} \subseteq (S\cup S^{-1})^{m_0}$$ and
$$S\cup S^{-1} \subseteq (T\cup T^{-1})^{m_1}.$$  This observation immediately gives the following important result.
\begin{thm} Let $G$ be a finitely generated group.  Then all growth functions of $G$ are asymptotically equivalent.
\end{thm}

We thus find it convenient to work with growth functions modulo the relation of asymptotic equivalence.  As a result, we can speak unambiguously in this sense of the \emph{growth function} of a finitely generated group $G$.

One can do a similar construction for algebras.  If $K$ is a field and $A$ is a finitely generated $K$-algebra, then we can select a finite-dimensional $K$-vector subspace $V$ of $A$ with the property that $1_A\in V$ and $V$ generates $A$ as a $K$-algebra.  We can again define a monotonically increasing function
\begin{equation} d_V:\mathbb{N}\to \mathbb{N},\end{equation}
given by the rule
\begin{equation} d_V(n) = {\rm dim}_K(V^n).
\end{equation}
As in the case with groups, if $V$ and $W$ are two finite-dimensional subspaces of $A$ that both contain $1$ and generate $A$ as a $K$-algebra then $d_V(n)\asymp d_W(n)$, and we can again speak unambiguously of the growth function of an algebra.

The growth function provides an important invariant in the study of algebras and groups.  The following definition gives a coarse, but nevertheless useful, taxonomy in the study of groups and algebras in terms of growth.
\begin{defn}Let $f(n)$ be a growth function of either a group or an algebra.  We say that the growth is \emph{polynomially bounded} if there exists a positive real number $\alpha>0$ such that 
$f(n)\preceq n^{\alpha}$; we say that the growth is \emph{exponential} if there exists $C>1$ such that
$f(n) \succeq C^n$; otherwise, we say that the growth is of \emph{intermediate type}.
\end{defn}

The Gelfand-Kirillov dimension of an algebra is a related invariant that gives one information about its growth.  If $A$ is an algebra of polynomially bounded growth, then the \emph{Gelfand-Kirillov} dimension of $A$ is the infimum over all positive numbers $\alpha$ such that its growth function is asymptotically dominated by $n^{\alpha}$.  If, on the other hand, $A$ is an algebra of superpolynomial growth, then we declare the Gelfand-Kirillov dimension of $A$ to be infinite. In particular, the example of Lenagan, Smoktunowicz, and the second-named author has a growth function that is asymptotically dominated by $n^{3+\epsilon}$ for every $\epsilon>0$.

One of the truly great results in the theory of growth of groups and algebras is Gromov's \cite{Gro} characterization of finitely generated groups  with polynomially bounded growth, showing that all such groups are nilpotent-by-finite.  In particular, this gives an affirmative answer to the Burnside problem for the class of finitely generated groups of polynomially bounded growth.  Golod and Sharafevich \cite{Go, GS} gave the first example of a finitely generated infinite torsion group and their example is easily seen to have exponential growth.  In light of the work of Gromov, it was natural to ask whether a finitely generated torsion group of intermediate growth could exist.  The first such example was given by Grigorchuk \cite{Gri}, who produced a torsion group with intermediate growth; in fact, the growth function of his example is aysmptotically greater than $\exp(\sqrt{n})$ but asymptotically less than $\exp(n)$.  Since Grigorchuk's original example, there have been improvements to the branch group construction, but at the moment there does not exist a family of constructions that can achieve arbitrarily slow superpolynomial growth.

Theorem \ref{thm: main1} actually shows that for any monotonically increasing function $\beta(n)$ of intermediate growth, there is an algebraic algebra over an uncountable field with the property that its growth function is asymptotically dominated by $\beta(n)$.  This is a sharp contrast with the corresponding situation for torsion groups.

\section{Combinatorial results}
In this section, we modify the construction of Lenagan and Smoktunowicz and the second-named author \cite{LSY} to give a general criterion for producing algebras with low growth and certain presentations.

Let $K$ be a field and let $A=K\{x,y\}$ denote the free $K$-algebra on two generators $x$ and $y$.  
Then $A$ is an $\mathbb{N}$-graded algebra, and we let $A(n)$ denote the $K$-subspace spanned by all words over $x$ and $y$ of length $n$.
We make use of a key proposition of Lenagan and Smoktunowicz \cite[Theorem 3]{LS}, which we have expressed in a slightly more general form.
\begin{prop} Let $f,g:\mathbb{N}\to \mathbb{N}$ be two maps satisfying
\begin{enumerate}
\item[(i)] $f(i-1)< f(i)-g(i)-1$ for all natural numbers $i$,
\item[(ii)] for each natural number $i$, there is a subspace $W_i\subseteq A(2^{f(i)})$ whose dimension is at most $2^{2^{g(i)}}-2$,
\end{enumerate}
and let $$T=\bigcup_i \{f(i)-g(i) -1,f(i)-g(i),\ldots , f(i)-1\}.$$
Then for each natural number $n$, there exist $K$-vector subspaces $U(2^n)$ and $V(2^n)$ of $A(2^n)$ satisfying the following properties:
\begin{enumerate}
\item $U(2^n)\oplus V(2^n)=A(2^n)$ for every natural number $n$;
\item $\dim(V(2^n))=2$ whenever $n\not\in T$;
\item $\dim(V(2^{n+j} ))=2^{2^j}$ whenever $n= f(i)-g(i)-1$ and $0\le j\le g(i)$;
\item for each natural number $n$, $V(2^n)$ has a basis consisting of words over $x$ and $y$;
\item for each natural number $i$, $W_i\subseteq U(2^{f(i)})$;
\item $A(2^n)U(2^n)+U(2^n)A(2^n)\subseteq U(2^{n+1})$ for every natural number $n$;
\item $V(2^{n+1})\subseteq V(2^n)V(2^n)$ for every natural number $n$;
\item if $n\not \in T$ then there is some word $w\in V(2^n)$ such that $wA(2^n)\subseteq U(2^{n+1})$.
\end{enumerate}
\label{key}
\end{prop}
One should think of the subspaces $U(n)$ and $V(n)$ as follows.  Condition (6) says that the sum of the $U(n)$ is in some sense very close to being a two-sided ideal.  It is not a two-sided ideal, but we will show that there is a homogeneous two-sided ideal $I$ which is a close approximation to this space.  Then one should think of the image of the sum of the $V(n)$ when we mod out by this ideal as being very close to a basis for the factor ring $A/I$.   

The fact that there are infinitely many $n\not \in T$ and conditions (2) and (3) say that the growth of $A/I$ should be small if $g(n)$ is grows sufficiently slowly compared to $f(n)$. The role of the subspaces $W_i$ is that they correspond to homogeneous relations we introduce.  Thus if we are not introducing too many relations and we have that the dimension of $W_i$ is bounded by $2^{2^{g(i)}}-2$, then we can hope to find an infinite-dimensional algebra with slow growth in which the images of all relations coming from the subspaces $W_i$ are zero.  

\begin{proof}[Proof of Proposition \ref{key}] The proof is similar to that of Lenagan and Smoktunowicz and the second-named author \cite[Theorem 3.1]{LSY}.   We use induction on $n$ and divide the proof into three cases.  We take $V(1)=Kx+Ky$ and $U(1)=(0)$.  Then conditions (1)--(8) are satisfied, when applicable, in the case that $n=0$.  Next suppose that we have defined $V(2^m)$ and $U(2^m)$ for all $m\le n$ which satisfy conditions (1)--(8).  We show how to define $V(2^{n+1})$ and $U(2^{n+1})$.  
There are three cases.
\vskip 2mm
\emph{Case I:} $\{n,n+1\}\subseteq T$.  In this case, we take $V(2^{n+1})=V(2^n)V(2^n)$ and
$U(2^{n+1})=A(2^n)U(2^n)+U(2^n)A(2^n)$.  Then we only need to check that conditions (1) and (4) hold.  Since $V(2^n)$ has a basis consisting of words, we see that $V(2^{n+1})$ has a basis consisting of all words formed by concatenating two words from a basis for $V(2^n)$ consisting of words.   Since $V(2^n)\oplus U(2^n)=A(2^n)$, we see that
$$A(2^{n+1})=\left(V(2^n)\oplus U(2^n)\right)\left(V(2^n)\oplus U(2^n)\right)=V(2^{n+1})\oplus U(2^{n+1}),$$ giving condition (1).
\vskip 2mm
\emph{Case II:} $n\not\in T$.  Then $V(2^n)$ has a basis consisting of two words $w_1$ and $w_2$.   We set $V(2^{n+1})$ to be the span of $w_1^2$ and $w_1w_2$, and we take
$$U(2^{n+1})=A(2^n)U(2^n)+U(2^n)A(2^n)+w_2V(2^n).$$  Then by construction properties 
(1)--(8) hold (note that $w_2A(2^n)\subseteq U(2^{n+1})$, giving property (8)).
\vskip 2mm
\emph{Case III:} $n\in T$, $n+1\not\in T$.  This is the only case where the functions $f$ and $g$ come in to play.  We pick a basis $x_1,\ldots ,x_d$ for $W_i$.  Since $A(2^n)=V(2^n)\oplus U(2^n)$, we see that decompose each $x_i$ uniquely as $y_i+z_i$, where $y_i\in V(2^n)V(2^n)$ and $z_i\in U(2^n)A(2^n)+A(2^n)U(2^n)$.  Since $n$ is in $T$ and $n+1\not \in T$, we see that $n=f(i)-1$ for some $i$ and hence ${\rm dim}(V(2^n))\ge 2^{2^{g(i)}}$.   By assumption $d={\rm dim}(W_i)\le 2^{2^{g(i)+1}}$, which implies that $${\rm dim}(V(2^n)V(2^n))\ge {\rm dim}(W_i)+2.$$  Thus the span of the $y_i$ is a subspace of $V(2^n)V(2^n)$ whose codimension is at least $2$.  Moverover, $V(2^n)V(2^n)$ has a basis consisting of words over $x$ and $y$ and hence there are two words $w_1,w_2\in V(2^n)V(2^n)$ such that $\{w_1,w_2,y_1,\ldots ,y_d\}$ is a linearly independent set.  
We take $V(2^{n+1})$ to be the span of $w_1$ and $w_2$.  We pick a subspace $Y\subseteq A(2^{n+1})$ with the properties that $Y\oplus (Kw_1+Kw_2)=V(2^{n})V(2^n)$ and $Y\supseteq Ky_1+\cdots +Ky_d$.  Then we take
$$U(2^{n+1})=U(2^n)A(2^n) + V(2^n)U(2^n) + Y.$$
Then by construction we have properties (1)--(4) hold.  Also, property (5) holds, as $Y$ contains the $y_i$ and each $x_i$ is in $U(2^n)A(2^n) + V(2^n)U(2^n)$.  By construction (6)--(7) hold and (8) does not apply.  This completes the proof.
\end{proof}

We find it convenient to fix our notation for the remainder of this section.
\begin{notn} We use the following notation and assumptions:
\begin{enumerate} \item we let $K$ denote a field;
\item we let $A=K\{x,y\}$ denote the free $K$-algebra on two generators;
\item we assume that we have maps $f,g:\mathbb{N}\to \mathbb{N}$ and subspaces $W_1,W_2,\ldots $ with $W_i\subseteq A(2^{f(i)})$ satisfying conditions (i) and (ii) of Proposition \ref{key};
\item  for each natural number $n$, we assume that we have homogeneous subspaces $U(2^n)$ and $V(2^n)$ of $A(2^n)$ satisfying conditions (1)--(8) of Proposition \ref{key}.
\end{enumerate}
\label{notn: 1}
\end{notn}
Using the assumptions and notation of Notation \ref{notn: 1}, we introduce four auxiliary families of subspaces.
For each natural number $n$, we construct four subspaces $L(n), L'(n), R(n),$ and $R'(n)$ of $A(n)$ satisfying the following properties:
\begin{equation}
L(n)A(2^{m+1}-n)\subseteq U(2^{m+1})\qquad {\rm for}~ n\in \{2^m,\ldots , 2^{m+1}-1\}; \end{equation}

\begin{equation} A(2^{m+1}-n)R(n)\subseteq U(2^{m+1})\qquad {\rm for}~n\in \{2^m,\ldots , 2^{m+1}-1\}; \end{equation}
and
\begin{equation} L(n)\oplus L'(n)=R(n)\oplus R'(n)=A(n) \qquad \textrm{for every natural number }n. 
\end{equation}
We first construct $L(n)$ and $R(n)$.  Given a natural number $n$, we pick $m$ such that 
$2^m\le n<2^{m+1}$ and define
\begin{equation}L(n) \ =  \ \{ x\in A(n)~:~xA(2^{m+1}-n)\subseteq U(2^{m+1}),
\end{equation}
and
\begin{equation}R(n) \ =  \ \{ x\in A(n)~:~A(2^{m+1}-n)x\subseteq U(2^{m+1}).
\end{equation}
In the paper of Lenagan and Smoktunowicz, the spaces $L, L', R,$ and $R'$ are called, respectively, $R, Q, S,$ and $W$.  We change notation only because we find it convenient to let $L(n)$ denote the set of elements of $A(n)$ such that when we multiply on the \emph{left} by these elements, we land in $U(2^{m+1})$.  Similarly, $R(n)$ denotes the elements which when we multiply on the \emph{right} with these elements we land in $U(2^{m+1})$.  Then $L'(n)$ and $R'(n)$ are just complements of $L(n)$ and $R(n)$, respectively, which we shall choose later.

We note that many of the proofs of Lenagan and Smoktunowicz and the second-named author \cite{LSY} go through more or less unchanged in this section, although we have stated some of them in a slightly different form.   For this reason, and the fact that our choice of notation is somewhat different than that used in the aforementioned paper, we give proofs of these results.

\begin{lem} Assume the notation and assumptions of Notation \ref{notn: 1}.  Let $n$ be a natural number and let $2^{i_1}+\cdots + 2^{i_d}$ denote its binary expansion with $i_1<i_2<\cdots <i_d$.  Then $$V(2^{i_1})V(2^{i_2})\cdots V(2^{i_d})+R(n)=V(2^{i_d})V(2^{i_{d-1}})\cdots V(2^{i_1})+L(n)=A(n).$$
\label{lem: 11}
\end{lem}
\begin{proof} We first prove that 
$$V(2^{i_1})V(2^{i_2})\cdots V(2^{i_d})+R(n)=A(n).$$
We note that for $j\in \{1,\ldots ,d\}$, we have $A(2^{i_j})=U(2^{i_j})\oplus V(2^{i_j})$.  
We have the decomposition
$$A(n) = \prod_{j=1}^d A(2^{i_j}) = \prod_{j=1}^d \left(U(2^{i_j})\oplus V(2^{i_j})\right).$$
Consequently
$$A(n) =\left( \prod_{j=1}^d V(2^{i_j})\right) \oplus T(n),$$
where 
$$T(n)=\sum_{j=1}^n \left( \prod_{\ell<j} A(2^{i_{\ell}})\right) U(2^{i_{j}}) \left( \prod_{\ell>j} A(2^{i_{\ell}})\right).$$
By Proposition \ref{key} (6), we have that $A(2^m)U(2^m)$ and $U(2^m)A(2^m)$ 
are contained in $U(2^{m+1})$.  
A straightforward induction (cf. Lenagan et al. \cite[Lemma 3.2]{LSY}) gives that if $p>m$ and $0\le q< 2^{m-n}$ then \begin{equation}\label{eq: ind} 
A(q 2^m)U(2^m) A(2^p - (q+1)2^m)\subseteq U(2^p). \end{equation}  

Observe that for $j\in \{1,\ldots ,d\}$, the space
\begin{eqnarray*}
T_j(n) &:=& \left( \prod_{\ell<j} A(2^{i_{\ell}})\right) U(2^{i_{j}}) \left( \prod_{\ell>j} A(2^{i_{\ell}})\right) 
\\ &=&A(2^{i_1}+\cdots + 2^{i_{j-1}})U(2^{i_j})A(2^{i_{j+1}}+\cdots + 2^{i_d})\end{eqnarray*}
has the property that
\begin{eqnarray*}
A(2^{i_{d}+1}-n)T_j(n) &=& A(2^{i_d+1}-n+ 2^{i_1}+\cdots + 2^{i_{j-1}})U(2^{i_j})A(2^{i_{j+1}}+\cdots + 2^{i_d}) \\
&=& A(2^{i_d+1}-2^{i_j}-2^{i_{j+1}}-\cdots - 2^{i_d})U(2^{i_j})A(2^{i_{j+1}}+\cdots + 2^{i_d})\\
&=& A(q2^{i_j})U(2^{i_j})A(2^{i_d+1}-(q+1)2^{i_j}),\end{eqnarray*}
where $$q=2^{i_d+1-i_j}-\sum_{\ell=j}^d 2^{i_{\ell}-i_j}.$$
Hence from Equation (\ref{eq: ind}), we see that $A(2^{i_{d}+1}-n)T_j(n)\subseteq U(2^{i_d+1})$ and so by definition,
$T_j(n)\subseteq R(n)$.  Thus $T(n)\subseteq R(n)$ and so 
$$V(2^{i_1})V(2^{i_2})\cdots V(2^{i_d})+R(n)=A(n).$$

The fact that
$$V(2^{i_1})V(2^{i_2})\cdots V(2^{i_d})+L(n)=A(n)$$ follows from a symmetric argument, in which we decompose $A(n)$ in the reverse order as done in the first part of the proof.  That is, we write
$$A(n)=A(2^{i_d})\cdots A(2^{i_1}),$$ and proceed identically to how we argued in the first case. 

\end{proof}
\begin{cor} Assume the notation and assumptions of Notation \ref{notn: 1}.  Let $n$ be a natural number and let $2^{i_1}+\cdots + 2^{i_d}$ denote its binary expansion with $i_1<i_2<\cdots <i_d$.  Then there exist subspaces $$R'(n)\subseteq V(2^{i_1})V(2^{i_2})\cdots V(2^{i_d})$$ and $$L'(n)\subseteq V(2^{i_d})V(2^{i_{d-1}})\cdots V(2^{i_1})$$ such that
$$R(n)\oplus R'(n)=L(n)\oplus L'(n)=A(n).$$  
\end{cor}
\begin{proof} This is a straightforward consequence of Lemma \ref{lem: 11}.
\end{proof}
The next result combines the preceding results and shows that under general conditions one can construct a homomorphic image of a free algebra on two generators with good upper bounds on its growth.
\begin{prop} Assume the notation and assumptions of Notation \ref{notn: 1}.  There exists a homogeneous two-sided ideal $I$ of $A=K\{x,y\}$ such that:
\begin{enumerate}
\item if $J$ is a homogeneous ideal of $A$ satisfying $J\subseteq \sum_k A(k2^{f(i)})W_iA$ for some natural number $i$, then $J\subseteq I$;
\item $V(2^n)\not \subseteq I$ for each natural number $n$;
\item $I$ has infinite codimension;
\item for every natural number $n$, 
$${\rm dim}(A(n)/I(n)) \le \sum_{j=0}^n {\rm dim}(L'(j)){\rm dim}(R'(n-j)).$$
\end{enumerate}\label{prop: ideal}
\end{prop}
\begin{proof} Let $n$ be a natural number and let $m$ be the unique nonnegative integer satisfying
$$2^m\le n < 2^{m+1}.$$
We define a subset $I(n)$ of $A(n)$ by declaring that 
$x\in I(n)$ if 
\begin{equation} A(j)xA(2^{m+2}-j-n)\subseteq U(2^{m+1})A(2^{m+1})+A(2^{m+1})U(2^{m+1})\end{equation}
for every $j\in \{0,1,\ldots , 2^{m+2}-n\}$.
We then define
\begin{equation}
I \ := \ \bigoplus_{n=1}^{\infty} I(n)\ \subseteq \ A.
\end{equation}
The fact that $I$ is a two-sided ideal of $A$ follows exactly as given in the proof of Lenagan and Smoktunowicz \cite[Theorem 5]{LS}.

To prove (1), we let $i$ and $n$ be a natural numbers and let $m$ be the unique nonnegative integer satisfying
$2^m\le n < 2^{m+1}.$  
Suppose that $J$ is a two-sided homogeneous ideal of $A$ satisfying
$$J\subseteq \sum_k A(k2^{f(i)})W_iA.$$  It is sufficient to show that if $x\in J$ is a nonzero homogeneous element of degree $n$, then
$$A(j)xA(2^{m+2}-j-n)\subseteq U(2^{m+1})A(2^{m+1})+A(2^{m+1})U(2^{m+1})$$ for every
$j\le 2^{m+2}-n$.  By assumption, $$A(j)xA(2^{m+2}-j-n)\subseteq J \subseteq \sum_k A(k2^{f(i)})W_iA.$$
Since every element of $A(j)xA(2^{m+2}-j-n)$ has degree $2^{m+2}$, we see that 
$$A(j)xA(2^{m+2}-j-n)\subseteq J \subseteq \sum_{k=0}^{2^{m+2-f(i)}-1} A(k2^{f(i)})W_iA(2^{m+2}-(k+1)2^{f(i)}).$$
By Property (5) of Proposition \ref{key}, we have
$$A(k2^{f(i)})W_iA(2^{m+2}-(k+1)2^{f(i)}) \subseteq A(k2^{f(i)})U_iA(2^{m+2}-(k+1)2^{f(i)})$$
for $k\in \{0,1,\ldots , 2^{m+2-f(i)}-1\}$.  
We now consider two cases.  If $k<2^{m+1-f(i)}$, then
\begin{eqnarray*}
&~& A(k2^{f(i)})U_iA(2^{m+2}-(k+1)2^{f(i)})\\
&=& A(k2^{f(i)})U_iA(2^{m+1}-(k+1)2^{f(i)})A(2^{m+1})\subseteq U(2^{m+1})A(2^{m+1}),
\end{eqnarray*}
by Equation (\ref{eq: ind}).  If $2^{m+1-f(i)}\le k < 2^{m+2-f(i)}$, then
\begin{eqnarray*}
&~& A(k2^{f(i)})U_iA(2^{m+2}-(k+1)2^{f(i)})
\\ &=& A(2^{m+1})A(k2^{f(i)} -2^{m+1})U_iA(2^{m+2}-(k+1)2^{f(i)})\subseteq A(2^{m+1})U(2^{m+1}),
\end{eqnarray*}
again by Equation (\ref{eq: ind}).  Thus we see that
$$\sum_{k=0}^{2^{m+2-f(i)}-1} A(k2^{f(i)})W_iA(2^{m+2}-(k+1)2^{f(i)})\subseteq U(2^{m+1})A(2^{m+1})+A(2^{m+1})U(2^{m+1}),$$
and so we see that $J\subseteq I$, as required.  

To obtain (2), let $n$ be a natural number.  We claim that $V(2^n)\not\subseteq I$. By Proposition \ref{key}, $V(2^n)$ is spanned by words over $x$ and $y$ and $V(2^{n+2})\subseteq V(2^n)^4$.  Moreover, if one examines the construction of $V(2^j)$ for a natural number $j$, we see that there are nonzero words $w_0,w_1,w_2,w_3$ over $x$ and $y$ in $V(2^n)$ such that
$w:=w_0w_1w_2w_3\in V(2^{n+2})$.  Suppose that $V(2^n)\subseteq I$.  Then $w_0\in I$ and so by our definition of $I$, $$w=w_0(w_1w_2w_3)\in U(2^{n+1})A(2^{n+1})+A(2^{n+1})U(2^{n+1})\subseteq U(2^{n+2}).$$
On the other hand, $w\in V(2^{n+2})$, by assumption, and so $w\in V(2^{n+2})\cap U(2^{n+2})=(0)$, a contradiction.  Thus we see that $V(2^n)\not\subseteq I$.

To see (3), note that if $I$ had finite codimension, then we would have $A(2^n)\subseteq I$ for some natural number $n$, which contradicts the fact that $V(2^n)\not\subseteq I$.  

The proof of (4) is identical to the proof given by Lenagan, Smoktunowicz and the second-named author \cite[Theorem 5.1]{LSY}.
\end{proof}
The final result we need is a growth estimate.
\begin{prop} \label{prop: est}
Assume the notation of the statement of Proposition \ref{key} and let $n$ be a natural number.  Then
$${\rm dim}(V(1)V(2)\cdots V(2^n)) \le  2^{2n} 2^{2^{g(1)}+\cdots + 2^{g(i)}}$$ and 
$${\rm dim}(V(2^n)V(2^{n-1})\cdots V(1)) \le  2^{2n} 2^{2^{g(1)}+\cdots + 2^{g(i)}},$$
where $i$ is the nonnegative integer satisfying $f(i) \le n <f(i+1)$.
\end{prop}
\begin{proof} We pick $i$ such that $f(i) \le n <f(i+1)$.
Then
\begin{eqnarray*}
{\rm dim}(V(1)V(2)\cdots V(2^n))  &\le & \left( \prod_{j\not \in T} 2\right) \cdot \prod_{\ell=1}^i \prod_{a=0}^{g(\ell)} V(2^{f(\ell)-g(\ell)-1+a}) \\
&\le & 2^n \cdot  \prod_{\ell=1}^i \prod_{a=0}^{g(\ell)} 2^{2^a} \\
&\le & 2^n \cdot  \prod_{\ell=1}^i 2^{2^{g(\ell)+1}}\\
&=& 2^{n+i} 2^{2^{g(1)}+\cdots +2^{g(i)}}\\
&\le & 2^{2n} 2^{2^{g(1)}+\cdots +2^{g(i)}}.
\end{eqnarray*}
The other inequality follows in the same manner.
\end{proof}
\section{Construction}
In this section, we prove Theorem \ref{thm: main1}.  In order to do this, we will require a combinatorial lemma that allows us to get around the problems inherent in working over an uncountable base field.  

\begin{lem} Let $K$ be a field, let $n$, $p$, and $d$ be natural numbers, and let $W$ be a $d$-dimensional subspace of $K\{x,y\}$ spanned by $d$ non-trivial words over $x$ and $y$ whose lengths are uniformly bounded by $p$.   Then there is a subspace $Y$ of $A(n)$ whose dimension is at most $$(n+1)^d(2p^2)4^p$$ such that for all sufficiently large $j$, the ideal generated by all $j$th powers of elements in $W$ is contained in the right ideal
$$\sum_{k=0}^{\infty} A(k n) Y A.$$
\label{lem: smok}
\end{lem}
\begin{proof}
We write $W={\rm Span}\{w_1,\ldots ,w_d\}$, where $w_1,\ldots ,w_d$ are words over $x$ and $y$.  
For each natural number $m\le n$ and $j< 2p$, we let
$\mathcal{E}(m,j)$ denote the set of all sequences $(i_1,\ldots ,i_d)\in 
\left( \mathbb{Z}_{\ge 0}\right)^d$ 
with $$i_1+\cdots +i_d=m$$ and $$i_1|w_1|+\cdots + i_d|w_d|=n-j,$$ where $|w|$ denotes the length of a word $w$ over the alphabet $\{x,y\}$.  

Given $(i_1,\ldots ,i_d)\in \mathcal{E}(m,j)$, we let $C(i_1,\ldots ,i_d)\in A$ denote the coefficient of the term with monomial $t_1^{i_1}\cdots t_d^{i_d}$ in $(w_1 t_1+\cdots +w_d t_d)^m \in A[t_1,\ldots ,t_d]$. 
Then $C(i_1,\ldots ,i_d)$ is a homogeneous element of degree $n-j$ in $A$.

We then let 
$$Y:=\sum_{\stackrel{1\le m\le n}{j< 2p}} ~~ \sum_{i\le \max(j,p-1)} ~~ \sum_{(i_1,\ldots ,i_d)\in \mathcal{E}(m,j)} A(i)C(i_1,\ldots ,i_d)A(j-i).$$
Since there are at most $(n+1)^{d-1}$ elements of $\mathcal{E}(m,j)$ we see that the dimension of $Y$ is at most 
$$(n+1)^d(2p^2)4^p.$$
It remains only to show that there exists some $m$ such that 
$$Ay^mA \subseteq \sum_{k=0}^{\infty} A(k n) Y A$$ for all $y\in W$.  To see this, we take $m=2n$ and let
$y=\lambda_1 w_1+\cdots + \lambda_d w_d\in W$ with $(\lambda_1,\ldots ,\lambda_d)\in K^d$.  
To show that $$Ay^{2n}A \subseteq \sum_{k=0}^{\infty} A(k n) Y A,$$ it is sufficient to show that
$$ay^{2n} \in  \sum_{k=0}^{\infty} A(k n) Y A$$ whenever $a$ is a word over $x$ and $y$.  
We may also assume that $|a|<n$.  
Given a natural number $\ell$, we let $T_{\ell}$ denote the collection of all words of the form $w_{i_1}w_{i_2}\cdots w_{i_{2n}}$ with $i_1,\ldots ,i_{2n}\in \{1,\ldots ,d\}$ such that $$|a|+|w_{i_1}\cdots w_{i_{\ell-1}}|<n\le |a|+|w_{i_1}w_{i_2}\cdots w_{i_{\ell}}|.$$  Note that $T_{\ell}$ is empty if $\ell$ is greater than $n-1$ and that $\{T_1,\ldots ,T_{n-1}\}$ forms a partition of the set of words of the form $w_{i_1}w_{i_2}\cdots w_{i_{2n}}$; moreover, if $w_{i_1}w_{i_2}\cdots w_{i_{2n}}\in T_{\ell}$ then $|a|+|w_{i_1}\cdots w_{i_{\ell}}|<n+p$.

Given a word of the form $w=w_{i_1}\cdots w_{i_m}$, we define
\begin{equation}\lambda(w):=\lambda_{i_1}\cdots \lambda_{i_m}.
\end{equation}
Then
$$ay^{2n} = \sum_{\ell=1}^{n-1} \sum_{w_{i_1}\cdots w_{i_{2n}}\in T_{\ell}} \lambda(w_{i_1}\cdots w_{i_{\ell}}) aw_1\cdots w_{i_{\ell}} y^{2n-\ell}.$$
Note that for $w=w_{i_1}\cdots w_{i_{2n}}\in T_{\ell}$, we may write $aw_{i_1}\cdots w_{i_{\ell}}$ as $w_0u$ for some word $w_0$ of length $n$ and some word $u$ of length $b$ for some $b$ strictly less than $p$.
Then it is sufficient to show that $$uy^{2n-\ell}\in\sum_{k=0}^{\infty} A(k n) Y A$$ for $\ell<n$.

We fix $\ell<n$ and a word $u$ of length $b<p$.  We let 
$X_{j}$ denote the collection of all words of the form $w_{i_1}w_{i_2}\cdots w_{i_{2n-\ell}}$ 
with the property that 
$$|w_{i_1}\cdots w_{i_{j}}|<n-b\le |w_{i_1}w_{i_2}\cdots w_{i_{j+1}}|.$$ Then $X_j$ is empty if $j$ is greater than $n-b-1$, and $\{X_1,\ldots ,X_{n-b-1}\}$ form a partition of the set of all words of the form $w_{i_1}\cdots w_{i_{2n-\ell}}$.  

Then
$$uy^{2n-\ell} = \sum_{j=1}^n \sum_{w\in X_{j}} \lambda(w)uw.$$
Observe that if $w_{i_1}\cdots w_{i_{2n-\ell}} \in X_{j}$, then $w_{i_{\sigma(1)}}\cdots w_{i_{\sigma(2n-\ell)}} \in X_{j}$ for any permutation $\sigma\in S_{2n-\ell}$ that fixes all natural numbers $>j$. 
We say that two elements of $X_j$ are equivalent if there exists such a permutation relating the two elements.

We note that if $\mathcal{C}$ is the equivalence class in $X_j$ containing the word $w_{i_1}\cdots w_{i_{2n-\ell}}$, then
\begin{eqnarray*}
&~& \sum_{w \in \mathcal{C}} \lambda(w)uw\\
&=& \sum_{\sigma\in S_j} \lambda(w_{i_1}\cdots w_{i_j})uw_{\sigma(i_1)}\cdots w_{\sigma(i_j)}
\cdot \lambda(w_{i_{j+1}}\cdots w_{2n-\ell})w_{i_{j+1}}\cdots w_{2n-\ell}\\
&=&  uC(a_1,\ldots ,a_d)\lambda(w_{i_{j+1}}\cdots w_{2n-\ell})w_{i_{j+1}}\cdots w_{2n-\ell},
\end{eqnarray*}
where $a_m$ is the cardinality of the set $\{e\le j~:~i_e=m\}$; i.e., $a_m$ is the number of occurrences of $w_m$ in $w_{i_1}\cdots w_{i_j}$.   By construction, we see that
$\sum_{w \in \mathcal{C}} \lambda(w)uw\in YA$.  Summing over all equivalence classes, we see that $\sum_{w\in X_{j}} \lambda(w)uw\in YA$ and hence
$$uy^{2n-\ell} = \sum_{j=1}^n \sum_{w\in X_{j}} \lambda(w)uw\in YA.$$
The result follows.
\end{proof}

We are now ready to prove our main theorem.

\begin{proof}[Proof of Theorem \ref{thm: main1}] We let $A=K\{x,y\}$ and we let $\mathcal{S}$ denote the countable collection consisting of all finite subsets of finite words over $x$ and $y$.  Given a finite subset $S=\{w_1,\ldots ,w_d\}\in \mathcal{S}$, we define
$${\rm deg}(S):=\max_{1\le i\le d} {\rm length}(w_i).$$  We note that we can enumerate the elements of $\mathcal{S}$ and we let $S_1,S_2,\ldots $ be an enumeration.

We say that a weakly increasing function $f: \mathbb{N}\to\mathbb{N}$ is \emph{sparse} if
\begin{enumerate}
\item $f(1) > g(1) + 1$, and
\item for every natural number $i$, $f(i+1) > g(i+1) + f(i) + 1$ and $g(i+1) > g(i)$,
\end{enumerate}
where $g: \mathbb{N}\to\mathbb{N}$ is defined by
$$g(i) := \left\lceil \log_2 \log_2 \left((2^{f(i)} + 1)^{{\rm card}( S_i)}(2 \deg(S_i)^2)4^{\deg(S_i)} + 2\right)\right\rceil$$ for each natural number $i$.

It is straightforward to show that sparse sequences of arbitrarily rapid growth exist; indeed, if $f(1),\ldots ,f(j)$ are defined and satisfy conditions (1) and (2) for all $i<j$, then there exists a natural number $N_0$ such that if $f(j+1)>N_0$ then condition (2) is satisfied for $i=j$. 

The sparseness condition allows us to use Lemma \ref{lem: smok} to prove the existence of a subspace $W_i\subseteq A(2^{ f(i) } )$ of dimension at most $2^{2^{g(i)}} - 2$ and a natural number $j$ such that
every element $y$ in the span of $S_i$ satisfies
\begin{equation} \label{eq: nil}
Ay^jA\subseteq \sum_k A(k2^{f(i)})W_iA.\end{equation}
We choose a collection of subspaces $U(2^n)$ and $V(2^n)$ of $A(2^n)$ for each natural number $n$, which satisfy conditions (1)--(8) of Proposition \ref{key}, where we use the functions $f$, $g$, and the subspaces $W_i$ chosen above.  
 
By Lemma \ref{lem: 11} and Proposition \ref{prop: ideal}, there exist a homogeneous two-sided ideal $I$ 
and subspaces
$$L'(n), R'(n)\subseteq V(2^{i_1})V(2^{i_2})\cdots V(2^{i_d})$$
such that:
\begin{enumerate}
\item if $J$ is an ideal of $A$ satisfying $J\subseteq \sum_k A(k2^{f(i)})W_iA$ for some natural number $i$, then $J\subseteq I$;
\item $V(2^n)\not \subseteq I$ for each natural number $n$;
\item $I$ has infinite codimension;
\item for every natural number $n$, 
$${\rm dim}(A(n)/I(n)) \le \sum_{j=0}^n {\rm dim}(L'(j)){\rm dim}(R'(n-j)).$$
\end{enumerate}
Note that the first condition along with Equation (\ref{eq: nil}) implies that $A/I$ is algebraic over $K$ and in particular every element of $A$ in the homogeneous maximal ideal is nil mod $I$.  The one remaining issue is the growth $A/I$.

Let $m$ be such that $2^m\le n<2^{m+1}$.  
Then, using Proposition \ref{prop: est}, we have
\begin{eqnarray*}
{\rm dim}(A(n)/I(n)) &=& \sum_{j=0}^n {\rm dim}(L'(j)){\rm dim}(R'(n-j)) \\
&\le & \sum_{j=0}^n \left( {\rm dim}(V(1)V(2)\cdots V(2^m))\right)^2 \\
&\le & 2^{4m} \left(2^{2^{g(1)}+\cdots +2^{g(i)}}\right)^2\\
&\le & n^4  2^{2^{g(i)+2}},\\ 
\end{eqnarray*}
where $i$ satisfies $f(i) \le m <f(i+1)$.
By assumption,
$$2^{2^{g(i)+2}}= \left(2^{2^{g(i)-1}}\right)^8 \le  \left((2^{f(i)} + 1)^{{\rm card}( S_i)}(2 \deg(S_i)^2)4^{\deg(S_i)} + 2\right)^8.$$
Consequently,
\begin{eqnarray*}
{\rm dim}(A(n)/I(n)) &\le &
 n^4 \left((2^{f(i)} + 1)^{{\rm card}( S_i)}(2 \deg(S_i)^2)4^{\deg(S_i)} + 2\right)^8\\
 &\le &  n^4 \left((2^{2f(i)})^{{\rm card}( S_i)}(2 \deg(S_i)^2)4^{\deg(S_i)} \right)^8\\
&\le & n^4 2^{16m\cdot {\rm card} (S_i)}\left(\deg(S_i) 2^{\deg(S_i)+1}\right)^{16}\\
&\le & n^{4 + 16{\rm card} (S_i)}(\deg(S_i) 2^{\deg(S_i)+1})^{16}.
\end{eqnarray*}

Recall that $\alpha: [0,\infty)\to[0,\infty)$ is a weakly increasing function that tends to $\infty$. 
Since there exist sparse sequences of arbitrarily fast growth, we can select a sparse sequence $f(i)$ that satisfies the conditions:
\begin{enumerate}
\item $\alpha\left(2^{f(i)}\right) > 17{\rm card}( S_i)$;
\item $f(i) > (\deg(S_i) 2^{\deg(S_i)+1})^{16}$.
\end{enumerate}
Then we have
\begin{eqnarray*}
{\rm dim}(A(n)/I(n)) &\le& n^{4 + 16{\rm card} (S_i)}(\deg(S_i) 2^{\deg(S_i)+1})^{16}\\
&\le& n^{4 + 16 \cdot \alpha\left(2^{f(i)}\right)/17}f(i)\\
&\le& n^{4+ 16\alpha(2^m)/17}\cdot m \\
&\le&  n^{5 + 16\alpha(n)/17}
\end{eqnarray*}
Since $\alpha(n) \to \infty$ as $n\to\infty$, we see that $\dim A(n)/I(n) \le n^{\alpha(n)}$ for all sufficiently large $n$.  Letting $B=A/I$, we obtain the desired result.
\end{proof}
We note that this proof does not work if $\alpha(n)$ does not tend to $\infty$, since the limit supremum of the cardinalities of the sets $S_i$ is infinite and thus it would not be possible to choose a function $f$ satisfies the condition $\alpha(2^{f(i)}) > 17 {\rm card} (S_i)$.  This shows that we cannot obtain polynomially bounded growth by our methods.
\section{Concluding remarks and questions}
Much is now known about Kurosh's problem with growth restrictions.  There are, however, several very important questions that remain.  In this section, we give what we feel are the most important remaining questions, some of which appear in the work of Lenagan, Smoktunowicz, and the second-named author \cite{LSY}.

The first question in quite natural in view of Theorem \ref{thm: main1}. 
\begin{quest} Let $K$ be an uncountable field.  Does there exist a finitely generated algebraic $K$-algebra of finite Gelfand-Kirillov dimension?
\end{quest}
To answer this question, genuinely new techniques will be required.  As we have pointed out, there are certain subtleties which appear when working over an uncountable base field, which make this problem very difficult.

As for the countable case, much more is known.  There are, however, several important open questions.
\begin{quest} Does there exist a real number $\alpha>0$ such that for each $\beta\in [\alpha,\infty)$ there exists a finitely generated algebraic algebra of Gelfand-Kirillov dimension $\beta$?
\end{quest}
Currently, only upper bounds on the growth have been obtained.  There are many results in the literature that show that for many classes of algebra, one can find an element in the class whose Gelfand-Kirillov dimension is precisely $\beta$ for each $\beta\in [2,\infty)$ (see, for example, Vishne \cite{Vi}).

Bergman \cite[Theorem 2.5, p. 18]{KL} showed that Gelfand-Kirillov dimension has a gap: there do not exist algebras with Gelfand-Kirillov dimension strictly between zero and one.  In fact, he showed that if $V$ is a finite-dimensional vector subspace that generates a finitely generated algebra $A$ and contains $1$, then either there is a positive constant $C>0$ such that
$${\rm dim}(V^n) \ < \ Cn$$ for all $n\ge 1$, or $${\rm dim}(V^n)\ge {n+2\choose 2}$$ for every $n\ge 1$.  
In light of this result, the class of algebras of quadratic growth form a natural boundary between linear growth and super-linear growth.  We recall that a finitely generated algebra $A$ over a field $K$ has \emph{quadratic growth} if there are a finite-dimensional $K$-subspace $V$ of $A$ that contains $1$ and generates $A$ as a $K$-algebra
and positive constants $C_0, C_1>0$ such that
$$C_0 n^2 \ \le \ {\rm dim}(V^n) \ \le \ C_1 n^2$$ for every $n\ge 1$.  Algebras of quadratic growth form a well-behaved subclass of algebras of Gelfand-Kirillov dimension two.  We believe that there do not exist finitely generated infinite-dimensional algebraic algebras of quadratic growth, although we have no evidence to support this belief.  We thus pose the following question.
\begin{quest} Does there exist a finitely generated infinite-dimensional algebraic algebra of quadratic growth?
\end{quest} 
The class of algebras of Gelfand-Kirillov dimension two is much more pathological than quadratic growth (see, for example, the paper of the first-named author \cite{Be}).  We thus pose the following question separately.
\begin{quest}
Does there exist a finitely generated infinite-dimensional algebraic algebras of Gelfand-Kirillov dimension two?
\end{quest}
\section*{Acknowledgments}
We thank Lance Small and Efim Zelmanov for many helpful comments and suggestions.

\end{document}